\documentclass[12pt]{article}

\usepackage[showframe=false]{geometry}                
\usepackage{graphicx}                
\usepackage[T1]{fontenc}             
\usepackage{tgpagella}               
\usepackage[symbol]{footmisc}        
\usepackage{hyperref}                
\usepackage{amsmath}                 
\usepackage{amssymb}
\usepackage{amsthm}                  
\usepackage{rotating}                
\usepackage{multirow}                
\usepackage{booktabs}                
\usepackage{adjustbox}               
\usepackage{mathtools}
\usepackage{xcolor}                  
\usepackage{colortbl}                
\usepackage{soul}
\usepackage{changepage}              
\usepackage[backend=biber,giveninits=false,uniquename=false, maxnames=99, uniquelist=false]{biblatex} 
\usepackage{caption}
\usepackage{datetime}                
\usepackage{authblk}                 

\newcommand{\nint}[1]{\left\|#1\right\|}      

\newcommand{\nexteq}{\displaybreak[0]\\ &=}

\newcommand{\refby}[1]{&&\text{(by (\ref{#1}))}}

\newcommand{\N}{\mathbb{N}}
\newcommand{\Z}{\mathbb{Z}}

\newcommand{\C}{\mathbb{C}}

\newcommand{\z}{\zeta}
\newcommand{\ph}{\varphi}

\DeclareMathOperator{\re}{Re}
\DeclareMathOperator{\im}{Im}

\newtheorem{thm}{Theorem}

\newtheorem{lem}{Lemma}
\newtheorem{prop}[lem]{Proposition}
\theoremstyle{definition}

\definecolor{cl0}{HTML}{cccccc}
\definecolor{cl1}{HTML}{62a0ea}
\definecolor{cl2}{HTML}{79c600}

\newdateformat{monthyeardate}{%
  \monthname[\THEMONTH] \THEYEAR}

\title{The Minimal Absolute Value of Sums of Fifth Roots of Unity}
\author[1,2]{Akihiro Munemasa}
\author[1,3]{Guillermo Núñez Ponasso}

\affil[1]{\footnotesize Graduate School of Information Sciences, Tohoku University, Sendai, Miyagi, Japan}
\affil[2]{\footnotesize
School of Science, China University of Geosciences, Beijing, China
}
\affil[3]{\footnotesize Dept. of Electrical \& Computer Engineering, Worcester Polytechnic Institute, Worcester, MA, USA}

\date{July 1, 2026}
\begin{document}

\maketitle

\abstract{
We determine the minimal absolute value of a non-vanishing sum of $n$ fifth roots of unity chosen with repetition, and characterize the corresponding sums. As a function of $n$, the minimal absolute value is monotone non-increasing over congruence classes of $n$ modulo $5$ and its only jumps occur when $n=5F_m$, $n=L_m$, or $n=2L_m$, where $F_m$ and $L_m$ denote the $m$-th Fibonacci and Lucas numbers respectively. To prove our results we reduce the problem to a series of inequalities involving rational approximations of the golden ratio $\ph=(1+\sqrt{5})/2$, the solutions of which can be characterized using the theory of continued fractions.
}

\section{Introduction}\label{sec:intro}
Let $\ell>1$ be an integer, and $\z_{\ell}$ be a primitive $\ell$-th root of unity. Let $\N=\{0,1,2,\dots\}$ be the set of non-negative integers. A weight $n$ sum of $\ell$-th roots of unity is an element of the set
\[\Sigma_{\ell}(n)=\left\{a_0+a_1\z_{\ell}+\dots+a_{\ell-1}\zeta_{\ell}^{\ell-1}:a_i\in\N,\ \sum_{i=0}^{\ell-1}a_i=n\right\},\]
We consider the problem of determining the minimal absolute value of a sum of $\ell$-th roots with weight $n$. More explicitly, the determination, for every integer $n\geq 1$, of the value
\[\sigma_{\ell}(n)=\min\{|s|:s\in\Sigma_{\ell}(n)\setminus\{0\}\}.\]
Note that the set $\Sigma_{\ell}(n)$ is finite, so the minimum is well-defined. Furthermore, the minimum $\sigma_{\ell}(n)$ does not depend on the choice of the primitive $\ell$-th root of unity $\z_{\ell}$. 

The minimal absolute value of a sum of $\ell$-th roots of unity (allowing vanishing sums) appears in a determinantal upper bound for matrices with entries in $\mu_{\ell}$ \cite{NP23, NP25}. Hence, the calculation of $\sigma_{\ell}(n)$ for all $n$ and a fixed $\ell$ is useful to the study of maximal determinant matrices.

Although determining $\sigma_{\ell}(n)$ for fixed $\ell$ and $n$ arbitrary is a natural question, the problem had not been previously considered in this form. A similar problem which received the attention of several authors \cite{B23, D18, KL00, M86} consists of finding asymptotic formulas for $\sigma_{\ell}(n)$ for $n$ fixed and $\ell$ arbitrarily large. This was first considered by Myerson \cite{M86}, and the counterpart to our results ~\textemdash~ with $n=5$ now fixed, and $\ell$ arbitrarily large ~\textemdash~ is given in \cite{B23}. An important related result is the proof by Lam and Leung \cite{LL00} that the set of weights $n$ for which $0\in\Sigma_{\ell}(n)$ is $ p_1\N + p_2\N +\dots + p_r\N$, where $\ell=p_1^{e_1}p_2^{e_2}\dots p_{r}^{e_r}$ is the factorization of $\ell$ into primes.

The minimal polynomial of $\z_{\ell}$ is linear or quadratic if and only if $\ell\in\{2,3,4,6\}$, and in all these cases it is easy to determine $\sigma_{\ell}(n)$ for all $n$.
\begin{prop}The value $\sigma_{\ell}(n)$ for $\ell\in\{2,3,4,6\}$ is given by the following formulas
\[\sigma_{2}(n)=\begin{cases}
    1 & \text{ if } 2\nmid n,\\
    2 & \text{ if } 2\mid n,
    \end{cases}\ \ \quad
\sigma_3(n)=\begin{cases}
    1 & \text{ if } 3\nmid n,\\
    \sqrt{3} & \text{ if } 3\mid n,
    \end{cases}\ \  \quad
\sigma_4(n)=\begin{cases}
    1 & \text{ if } 2\nmid n,\\
    \sqrt{2} & \text{ if } 2\mid n,
    \end{cases}\ \ 
\]
and $\sigma_6(n)=1$, for all $n\geq 1$.
\end{prop}
\begin{proof} The ring $R=\Z[\z_{\ell}]$ for $\ell\in\{2,3,4,6\}$ is a discrete subring of $\C$, so it is closed under taking powers. This means that there is no nonzero element with absolute value less than $1$. Hence, $\sigma_{\ell}(n)\geq 1$ in all cases. For $\ell=2,4$ there is a weight $n$ vanishing sum if and only if $n$ is even, and for $\ell=3$ there is a vanishing sum if and only if $3\mid n$. The corresponding formulas follow from the fact that the smallest absolute value of an element $z\in R$  with $|z|> 1$ is $2,\sqrt{3}=|1+2\z_3|,$ and $\sqrt{2}=|1+\z_4|$ for $\ell=2,3,$ and $4$ respectively. When $\ell=6$, we have trivially $\sigma_6(1)=1$, and since $|\z_6^2+\z_6^4|=1$, it follows that $\sigma_6(2)=1$. Noticing that $\z_6^3=-1$,  the sums $1+(1+\z_6^3)(n-1)/2$ for $n\geq 3$ odd, and $(\z_6^2+\z_6^4)+(1+\z_6^3)(n-2)/2$ for $n\geq 4$ even, show that $\sigma_6(n)=1$ for $n\geq 3$.
\end{proof}
In this paper we are concerned with the case $\ell=5$, i.e. we determine the minimal absolute value of a sum of $5$-th roots of unity of weight $n$. Since the minimal polynomial $\Phi_5(T)=T^4+T^3+T^2+T+1$ of $\z_{5}$ is quartic, $\ell=5$ constitutes the smallest case where the minimal polynomial of $\z_{\ell}$ has degree $>2$.

To state our main result, we introduce some notation: The Fibonacci sequence is given by the initial values $F_0:=0$, $F_1:=1$, and $F_{m}:=F_{m-1}+F_{m-2}$ for $m\geq 2$.
The $m$-th term $L_m$ of the Lucas sequence, for $m\geq 1$, is given as
\[
  L_m:=F_{m+1}+F_{m-1}=F_{m}+2F_{m-1}.
\]
Equivalently, the Lucas sequence can be defined by the linear recurrence $L_0:=2$, $L_1:=1$, and $L_{m}:=L_{m-1}+L_{m-2}$ for $m\geq 2$. Also note that $L_{m+1}+L_{m-1}=5F_{m}$ for $m\geq 1$.

We define two functions $\kappa_1(n)$ and $\kappa_2(n)$ for $n\not\equiv 0\pmod{5}$ by
\[
  \kappa_i(n):=\max \mathcal{K}_i(n),
\]
where
\[\mathcal{K}_i(n)=\{k: n\geq iL_{k} \text{ and } iL_{k}\equiv n\pmod{5}\},\]
i.e. $\kappa_1(n)$, and $\kappa_2(n)$, are the largest index of a Lucas number, or a twice Lucas number, congruent to $n$ modulo $5$ and bounded above by $n$, respectively. For example, $\kappa_1(7)=4$ because $L_4=7\leq 7$, and $\kappa_2(7)=2$ since $2L_2=6\leq 7$. Define:
\begin{equation}\label{eq:1c2}
  \kappa(n):=\max\{\kappa_1(n)-1,\kappa_2(n)\}.
\end{equation}
Extending the Lucas numbers to negative indices, if necessary, we can ensure that for all $n\geq 1$ the set $\mathcal{K}_1(n)\cup \mathcal{K}_2(n)$ is non-empty. If we let $\kappa_i(n):=-\infty$ whenever $\mathcal{K}_i(n)$ is empty, the number $\kappa(n)$ is well-defined and takes finite values for all $n\geq 1$.

We denote the golden ratio as

\[\ph:=\frac{1+\sqrt{5}}{2},\]
i.e. $\ph$ is the positive real root of the polynomial $T^2 -T-1$.

Our main results are the following:
\begin{thm}\label{thm:1a} Let $n\geq 1$ be an integer, $n\not\equiv 0\pmod{5}$, and  $\kappa$ be as in (\ref{eq:1c2}). Then
  \begin{equation}\label{eq:thm1}
    \sigma_5(n)=\frac{1}{\ph^{\kappa(n)}}.
  \end{equation}
\end{thm}
\begin{thm}\label{thm:2a} Let $k\geq 1$ be an integer, $n=5k$, and  $\nu(k) := \max\{m: F_{m}\leq k\}$. Then
  \begin{equation}\label{eq:thm2}
    \sigma_5(n)^2=\frac{\sqrt{5}}{\ph^{2\nu(k)-1}}.
  \end{equation}
\end{thm}
The values of $\sigma_5(n)$ for $n<35$ can be visualized in Table \ref{tab:1a}.
\begin{table}
  \centering
  \captionsetup{width=.92\linewidth}
  \caption{The value of $\sigma_5(n)$ for $n<35$. The jumps of $\sigma_5(n)$ on each congruence class modulo $5$ are indicated with a different color depending on the type of weight where they occur: $5F_m$, $L_m$, or $2L_m$.}\label{tab:1a}
  \begin{tabular}{l|*{7}c}
    $\sigma_5(n)$& 0 & $5$ & $10$ & $15$ & $20$ & $25$ & $30$\\ 
    \hline
    $+0$ & $-$ & \cellcolor{cl0}$\sqrt[4]{5}/\ph^{3/2}$& \cellcolor{cl0}$\sqrt[4]{5}/\ph^{5/2}$& \cellcolor{cl0}$\sqrt[4]{5}/\ph^{7/2}$& $\sqrt[4]{5}/\ph^{7/2}$& $\cellcolor{cl0}\sqrt[4]{5}/\ph^{9/2}$& $\sqrt[4]{5}/\ph^{9/2}$\\
    $+1$ & $1$ &\cellcolor{cl2} $1/\ph^2$ &\cellcolor{cl1}  $1/\ph^4$ &  $1/\ph^4$&  $1/\ph^4$&  $1/\ph^4$ &  $1/\ph^4$\\ 
    $+2$ & $1/\ph$& \cellcolor{cl1} $1/\ph^3$&  $1/\ph^3$&  $1/\ph^3$&\cellcolor{cl2}  $1/\ph^5$&  $1/\ph^5$&  $1/\ph^5$\\
    $+3$& \cellcolor{cl1}  $1/\ph$ & \cellcolor{cl2} $1/\ph^3$&  $1/\ph^3$& \cellcolor{cl1}  $1/\ph^5$&  $1/\ph^5$&  $1/\ph^5$&  $1/\ph^5$\\
    $+4$& \cellcolor{cl1}  $1/\ph^2$&  $1/\ph^2$&\cellcolor{cl2}  $1/\ph^4$&  $1/\ph^4$&  $1/\ph^4$&  \cellcolor{cl1} $1/\ph^6$&  $1/\ph^6$
  \end{tabular}
\end{table}



\section{Preliminaries}\label{sec:prelim}
To prove our results, we apply the theory of continued fractions. Reference material on this topic, and on the approximation of real irrational numbers by rational fractions, can be found in \cite{K} and \cite{RS}.

Suppose that $t>0$ is a real number with continued fraction expansion
\[
t = a_0+\frac{1}{a_1+\frac{1}{a_2+\frac{1}{\dots}}}
\]
where the $a_i$'s are integers, $a_0\geq 0$ and $a_i>0$ for $i\geq 1$. We write $t=[a_0;a_1,a_2,\dots]$. The $k$-th convergent of $t$ is the rational fraction,
\[\frac{A_k}{B_k}:=[a_0; a_1,\dots,a_k]:=a_0+\frac{1}{a_1+\frac{1}{\dots+\frac{1}{a_k}}},\]
where $B_k>0$. Formally defining
\[
\begin{bmatrix}
  A_{-1}\\
  B_{-1}
\end{bmatrix}
:=
\begin{bmatrix}
  1\\
  0
\end{bmatrix},
\]
the $k$-th convergent can be obtained via the following linear recursion
\[
\begin{bmatrix}
  A_{0}\\
  B_{0}
\end{bmatrix}
:=
\begin{bmatrix}
  a_0\\
  1
\end{bmatrix},
\quad \text{ and }
\begin{bmatrix}
  A_{k}\\
  B_{k}
\end{bmatrix}
=
a_k
\begin{bmatrix}
  A_{k-1}\\
  B_{k-1}
\end{bmatrix}
+
\begin{bmatrix}
  A_{k-2}\\
  B_{k-2}
\end{bmatrix},
\text{ for } k\geq 1.
\]
From this linear recursion, one can show that $\gcd(A_k,B_k)=1$ for $k\geq 0$. We define the $k$-th difference with respect to the irrational real number $t$ by
\[D_k:=B_kt-A_k.\]
Given an irrational real number $x$, we will denote the distance of $x$ to its nearest integer $[x]$ by
\[
  \nint{x} := |[x]-x|.
\]
Note that if $a$ is an integer, then $\nint{x+a}=\nint{x}$.

 Every positive integer $b$ has a unique expansion in terms of the denominators $B_k$ of the convergents of $t$ ~\textemdash~ this is known as the \textit{Ostrowski representation} of $b$ with respect to $t$.
\begin{thm}[Ostrowski's Algorithm \cite{O}, cf. Chapter II, \S 4 \cite{RS}]\label{thm:Ost}
  Suppose $t>0$ is irrational with continued fraction expansion $t=[a_0;a_1,a_2,\dots]$ satisfying $a_i\in\Z_{>0}$ for $i\geq 1$ and $a_0\in\N$. Let $B_k$ be the denominator of the $k$-th convergent of $t$. Then, every positive integer $b$ can be written uniquely as
  \[b=\sum_{k=0}^{R} c_{k+1}B_{k},\]
  where the $c_k$ satisfy $0\leq c_{k+1}\leq a_{k+1}$ for $k\geq 1$ and $0\leq c_1<a_1$, and for $k\geq 0$ if $c_{k+1}=a_{k+1}$ then $c_{k}=0$.
\end{thm}
The following theorem expresses $\|bt\|$ in terms of the differences $D_k$ whenever the Ostrowski representation of $b$ does not involve ``small indices'':

\begin{thm}[cf. Chapter II, \S 4 Theorem 1 (1) \cite{RS}]\label{thm:RS1} Let $t>0$ be an irrational number and $b>1$ an integer. Suppose the Ostrowski representation of $b$ with respect to $t$ is $b=\sum_{k=r}^Rc_{k+1}B_{k}$. If $r\geq 2$, then
  \[\|bt\|=\left|\sum_{k=r}^Rc_{k+1}D_{k}\right|.\]
\end{thm}

\begin{lem}\label{lem:0}
  Let $t>0$ be irrational and let $b\geq 1$ be an integer. Suppose the Ostrowski representation of $b$ with respect to $t$ is $b=\sum_{k=r}^R c_{k+1}B_{k}$, where $r\geq 2$. Then
  \[[bt]=\sum_{k=r}^Rc_{k+1}A_{k}.\]
\end{lem}
\begin{proof}
  Since $r\geq 2$, by Theorem~\ref{thm:RS1} we find that
  \[
   \|bt\|=\left|bt-\sum_{k=r}^R c_{k+1}A_{k}\right|.\\
  \]
  Since $t$ is irrational $\|bt\|< 1/2$. This implies the result.
\end{proof}

\begin{lem}[cf. Chapter II, \S 4 Lemma 1 \cite{RS}]\label{lem:RS1} Let $t>0$ be irrational and let $b\geq 1$ be an integer. Suppose the Ostrowski representation of $b$ with respect to $t$ is $b=\sum_{k=r}^R c_{k+1}B_{k}$. Then,
  \[
    |(c_{r+1}-1)D_{r}-D_{r+1}|<\left|\sum_{k=r}^R c_{k+1}D_{k}\right|<|c_{r+1}D_{r}-D_{r+1}|.
  \]
\end{lem}

From now on we will choose the primitive $5$-th root of unity $\z:=\exp(2\pi i/5)$, so that the following equations hold:
\begin{align*}
  \z+\z^4 &= \frac{1}{\ph},\\
  \z^2+\z^3 &= -\ph.
\end{align*}
Therefore, $\re(\z)=\re(\z^4)=1/(2\ph)$ and $\re(\z^2)=\re(\z^3)=-\ph/2$. From which we find that
\begin{align*}
  \im(\z)^2=\im(\z^4)^2 &= \frac{4\ph^2-1}{4\ph^2}=\ph^2\frac{4-\ph^2}{4},\\
  \im(\z^2)^2=\im(\z^3)^2 &=\frac{4-\ph^2}{4}.
\end{align*}
In particular,
\begin{equation}\label{eq:ratio}
  \frac{\im(\z)}{\im(\z^2)} = \ph.
\end{equation}
We also denote
\[
  \alpha := \ph + 2,
\]
since this constant will frequently appear in our arguments. 
The continued fraction expansion of $\ph$ is $\ph=[1;1,1,\dots]$, so $a_k=1$ for $k\geq 0$, and the convergents of $\ph$ are given by
\[
\begin{bmatrix}
  A_{k}\\
  B_{k}
\end{bmatrix}
=
\begin{bmatrix}
  F_{k+2}\\
  F_{k+1}
\end{bmatrix},
\text{ for } k\geq 1.
\]
The continued fraction expansion of $\alpha$ is $\alpha=\ph+2=[3;1,1,\dots]$, so $a_0=3$ and $a_k=1$ for $k\geq 1$. Hence, the convergents of $\alpha$ are given by
\[
\begin{bmatrix}
  A_{k}\\
  B_{k}
\end{bmatrix}
=
\begin{bmatrix}
  L_{k+2}\\
  F_{k+1}
\end{bmatrix},
\text{ for } k\geq 1.
\]
\begin{lem}\label{lem:1}
    If $b>0$ is an integer, then the Ostrowski representation  of $b$ with respect to $t=\ph$ or $\alpha$ expresses $b$ uniquely as a sum of non-consecutive Fibonacci numbers, i.e.
    \[b=\sum_{k=r}^Rc_{k}F_{k}\]
    where $r\geq 2$, $c_{k}\in \{0,1\}$, and $c_{k}c_{k+1}=0$ for $k\geq 2$. Furthermore, if $r\geq 3$, then
    \[
    [bt]=
    \begin{cases}
        \sum_{k=r}^Rc_{k}F_{k+1} & \text{if }t=\ph,\\
        \sum_{k=r}^Rc_{k}L_{k+1} & \text{if }t=\alpha.
    \end{cases}
    \]
\end{lem}
\begin{proof}
 Using the fact that $a_{k}=1$ for $k\geq 1$, the existence and uniqueness of the expression for $b$ is a consequence of Theorem~\ref{thm:Ost}. After re-indexing the sum, if $r\ge 3$, applying Lemma~\ref{lem:0}, the formula for $[bt]$ follows.
\end{proof}
In what follows, we always consider $t=\ph$ or $\alpha$. Since Lemma~\ref{lem:1} shows that Ostrowski representations coincide for both values of $t$, we will not specify $t$ unless necessary.

For $t=\ph$, it is easy to show by induction that for $k\geq -1$, the $k$-th difference is
\begin{equation*}
    D_{k}= F_{k+1}\ph - F_{k+2}=\frac{(-1)^{k}}{\ph^{k+1}}.
\end{equation*}
From this fact it follows that, for $t=\alpha$, the $k$-th difference is
\begin{equation}\label{eq:d2} D_{k}= F_{k+1}\alpha -L_{k+2}= F_{k+1}\ph - F_{k+2}=\frac{(-1)^k}{\ph^{k+1}}.
\end{equation}


\begin{lem}\label{lem:2} Let $t=\ph$ or $\alpha$. If $b>0$ is an integer with Ostrowski representation
  $b=\sum_{k=r}^{R}c_{k}F_{k}$,
 then
  \begin{equation}\label{eq:lem2}
    \|bt\|>\frac{1}{\ph^{r+1}}.
  \end{equation}
\end{lem}
\begin{proof} We may assume $c_{r}=1$. Reindexing, $b=\sum_{k=r-1}^{R-1} c_{k+1}F_{k+1}$. By Lemma~\ref{lem:RS1}, we find that
  \[|D_{r}|<\left|\sum_{k=r-1}^{R-1}c_{k+1}D_{k}\right|<|D_{r-1}-D_{r}|.\]
  By Lemma~\ref{lem:1}, $r\geq 2$ and we have
  \[|D_{r-1}-D_{r}|=\frac{1}{\ph^{r-1}}< 1.\]
  Hence $0< |\sum_{k=r-1}^{R-1} c_{k+1}D_{k}|< 1$, and
  \begin{align*}
    \|bt\| &= \left\|\sum_{k={r-1}}^{R-1}c_{k+1}D_{k}\right\|
          \nexteq \min\left\{\left|\sum_{k=r-1}^{R-1}c_{k+1}D_{k}\right|,1-\left|\sum_{k=r-1}^{R-1}c_{k+1}D_{k}\right|\right\}\\
           &>\min\{|D_{r}|,1-|D_{r-1}-D_{r}|\}\\
           &=\min\left\{\frac{1}{\ph^{r+1}},1-\left(\frac{1}{\ph^{r}}+\frac{1}{\ph^{r+1}}\right)\right\}\\
    &=\frac{1}{\ph^{r+1}}\min\{1,\ph^{r+1}-\ph^{2}\}\\
    &=\frac{1}{\ph^{r+1}}.\qedhere
  \end{align*}
\end{proof}

 Notice that since $1+\z+\z^2+\z^3+\z^4=0$, we have that $\Sigma_5(n)\subseteq \Sigma_5(n+5f)$ for all $f\geq 0$. Our arguments will be simpler if we restrict our attention to those sums of weight $n$ which cannot be represented by a sum of a smaller weight congruent to $n$ modulo $5$. In the following sections, we show that to prove Theorems \ref{thm:1a} and \ref{thm:2a}, it is sufficient to only consider sums belonging to the set
 \[\Sigma_5^0(n):=\{a_1\z+a_2\z^2+a_3\z^3+a_4\z^4: a_i\in\N,\  \sum_{i=1}^4 a_i=n\}.\]
 We write $\sigma_5^0(n):=\min\{|s|:s\in \Sigma_5^0(n)\}.$\\

The following lemma relates the minimal sum problem to two coupled Diophantine approximation problems, one for $\alpha$ and one for $\ph$. 
\begin{lem}\label{lem:bnxy}
    Suppose that  \[s=a_1\z+a_2\z^2+a_3\z^3+a_4\z^4\in\Sigma_5^0(n).\]
    Then,
    \[|\re(s)|=\frac{1}{2\ph}|n-b\alpha|,\]
    where $b:=a_2+a_3$, and
    \[|\im(s)|^2=\frac{4-\ph^2}{4}|x\ph- y|^2,\]
    where $x:=a_1-a_4$, and $y:=a_3-a_2$. Furthermore,  $1/\ph^{4}\geq |s|$ implies that $n=[b\alpha]$ and $b=[n/\alpha]$; and $1/\ph^3\geq |s|$ implies that $y=[x\ph]$ and $x=[y/\ph]$.
\end{lem}
\begin{proof}
    The calculation of $|\re(s)|$ is elementary. On the other hand we have by (\ref{eq:ratio})
  \begin{align*}
    |\im(s)|^2&=|(a_1-a_4)\im(\z)-(a_3-a_2)\im(\z^2)|^2\\
    &=\im(\z^2)^2|(a_1-a_4)\ph-(a_3-a_2)|^2\\
    &=\frac{4-\ph^2}{4}|x\ph-y|^2.
  \end{align*}
  Now, if $1/\ph^4\geq |s|\geq |\re(s)|$, then
    \[\frac{1}{2}>\frac{2}{\ph^3}\geq |n-b\alpha|.\]
  So, it follows that $n=[b\alpha]$ and $b=[n/\alpha]$. Similarly, if $1/\ph^3\geq |s|$, then
  \[\frac{1}{2}>\sqrt{\frac{4}{4-\ph^2}}\frac{1}{\ph^3}\geq |x\ph-y|.\]
  Hence $y=[x\ph]$ and $x=[y/\ph]$.
\end{proof}
\section{Proof of Theorem~\ref{thm:1a}}\label{sec:thm1}
Our proof strategy has three parts: (1) we exhibit sums achieving the claimed minimum; (2) we reduce the problem to showing a lower bound for sums in $\Sigma_5^0(n)$; (3) using Lemma~\ref{lem:bnxy}, the minimum sum problem is reduced into Diophantine inequalities in terms of convergents of $\alpha$ or $\ph$, which we resolve using Ostrowski representations ~\textemdash~ first bounding the real part, and then (if the real part does not determine $s$) the imaginary part.

Observe that
\[L_{m} \equiv \begin{cases}
    2 \pmod{5} & \text{ if } m\equiv 0\pmod{4},\\
    1 \pmod{5} & \text{ if } m\equiv 1\pmod{4},\\
    3 \pmod{5} & \text{ if } m\equiv 2\pmod{4},\\
    4 \pmod{5} & \text{ if } m\equiv 3\pmod{4}.\\
\end{cases}\]
Then, for all $m\geq 0$, $L_{m}\equiv 2L_{m+1}\equiv L_{m+4}\equiv 2L_{m+5}\not \equiv 0\pmod{5}$. This implies that for $n\not\equiv 0\pmod{5}$, $\kappa(n)=m$ if and only if
\begin{itemize}
    \item [(i)]  $n\equiv L_{m+1}\pmod{5}$ and $n\in[L_{m+1},2L_{m+2})$, or
    \item [(ii)] $n\equiv 2L_{m}\pmod{5}$ and $n\in [2L_{m},L_{m+3})$.
\end{itemize}
We split the proof into intervals of the type $[L_m,2L_{m+1})$, and $[2L_{m},L_{m+3})$.

\begin{lem}\label{lem:ub1}
  For all integers $n\geq 1$, $n\not\equiv 0\pmod{5}$,
  \[\sigma_5(n)\leq \frac{1}{\ph^{\kappa(n)}}.\]
\end{lem}
\begin{proof}
Suppose first that $n\equiv L_{m}\pmod{5}$ and $n\in [L_m,2L_{m+1})$ for some $m\geq 1$. Then $\kappa(n)=m-1$. Consider the sum
  \[s = f(1+\z+\z^2+\z^3+\z^4) + F_{m}+F_{m-1}\z^2+F_{m-1}\z^3,\]
  where $f$ is a non-negative integer and $5f=n-L_m$. Then $s$ has weight $5f + F_{m}+2F_{m-1}=n$ and by (\ref{eq:d2}),
  \begin{align*}
    |s| &= |F_{m}+F_{m-1}\z^2+F_{m-1}\z^3|=|F_{m}-F_{m-1}\ph|= \frac{1}{\ph^{m-1}}.
  \end{align*}
  Therefore,
  \[\sigma_5(n)\leq \frac{1}{\ph^{m-1}}=\frac{1}{\ph^{\kappa(n)}}.\]

  Now, suppose that $n\equiv 2L_{m}$ and $n\in [2L_m,L_{m+3})$. Then $\kappa(n)=m$. Write $5f=n-2L_m$ for some non-negative integer $f$ and consider the sum
  \[s = f(1+\z+\z^2+\z^3+\z^4) + F_{m+1}\z + F_{m-1}\z^2 + F_{m-1}\z^3 + F_{m+1}\z^4.\]
  Then $s$ has weight $5f+2F_{m+1}+2F_{m-1}=n$, and
  \begin{align*}
    |s| &= |F_{m+1}\z + F_{m-1}\z^2 + F_{m-1}\z^3 + F_{m+1}\z^4|\\
    &= \left|\frac{F_{m+1}}{\ph}- F_{m-1}\ph\right|\\
    &=\frac{1}{\ph} |L_{m}- F_{m-1}\alpha|\\
    &=\frac{1}{\ph^{m}}.\refby{eq:d2}
  \end{align*}
  Therefore,
  \[\sigma_5(n)\leq \frac{1}{\ph^m}=\frac{1}{\ph^{\kappa(n)}}.\qedhere\]
\end{proof}

\begin{prop}\label{prop:red} The inequality
\[\sigma_5^0(n)\geq \frac{1}{\ph^{\kappa(n)}},\]
for all integers $n\geq 1$ with $n\not\equiv 0\pmod{5}$, implies Theorem~\ref{thm:1a}.
\end{prop}
\begin{proof}
 By Lemma~\ref{lem:ub1}, to prove Theorem~\ref{thm:1a} it suffices to show that $\sigma_5(n)\geq 1/\ph^{\kappa(n)}$ for all $n\geq 1$, $n\not\equiv 0\pmod{5}$.
 Let 
 \[s=a_0+a_1\z+a_2\z^2+a_3\z^3+a_4\z^4\in\Sigma_5(n)\] be arbitrary. To show that $|s|\ge 1/\ph^{\kappa(n)}$ we may assume without loss of generality that $a_0=\min\{a_0,a_1,a_2,a_3,a_4\}$. Hence,
 \[s=(a_1-a_0)\z+(a_2-a_0)\z^2+(a_3-a_0)\z^3+(a_4-a_0)\z^4\in \Sigma_5^0(n-5a_0).\]
 Since $n-5a_0\equiv n\not\equiv 0\pmod{5}$, we have that $\kappa(n-5a_0)\leq \kappa(n)$ and thus
 \[|s|\geq \sigma_5^0(n-5a_0)\geq  \frac{1}{\ph^{\kappa(n-5a_0)}}\geq \frac{1}{\ph^{\kappa(n)}}.\qedhere\]
\end{proof}
   First, we consider the case where $n\equiv L_m\pmod{5}$ and $L_{m}\leq n <2L_{m+1}$. In particular $L_m\leq n\leq 2L_{m+1}-5$. In view of Lemma~\ref{lem:bnxy}, we determine an interval for $b=[n/\alpha]$ that corresponds to the given interval for $n$.
\begin{lem}\label{lem:0a} Let $m\geq 1$ be an integer. If  $L_{m}\leq n\leq 2L_{m+1}-5$, then
  \[F_{m-1}\leq \left[\frac{n}{\alpha}\right]<2F_{m}.\]
\end{lem}
\begin{proof}
  From Equation (\ref{eq:d2}),
  \[\frac{L_m}{\alpha}=F_{m-1}+\frac{(-1)^{m-1}}{\ph^{m-1}\alpha}.\]
  Since $n\geq L_{m}$, we have
  \[\frac{n}{\alpha}\geq F_{m-1}+\frac{(-1)^{m-1}}{\ph^{m-1}\alpha}>F_{m-1}-\frac{1}{\alpha}>F_{m-1}-\frac{1}{2},\]
  hence $F_{m-1}\leq [n/\alpha]$. On the other hand, from $n\leq  2L_{m+1}-5$ we have
  \begin{align*}
    \frac{n}{\alpha} &\leq 2\frac{L_{m+1}}{\alpha}-\frac{5}{\alpha}\\
    &= 2\left(F_{m}+\frac{(-1)^{m}}{\ph^{m}\alpha}\right)-\frac{5}{\alpha}\\
    &<2F_{m}-\frac{1}{2}.
  \end{align*}
  Therefore, $[n/\alpha]<2F_{m}$.
\end{proof}
The following result can be used to characterize the real part of the sums of weight $n$ with small absolute value.
\begin{lem}\label{lem:1a} Let $m\geq 7$ and let $b$ be an integer satisfying $F_{m-1}\leq b<2F_{m}$. If $[b\alpha]\equiv L_{m}\pmod{5}$ and
  \[\|b\alpha\|\leq \frac{2}{\ph^{m-2}},\]
  then $b=F_{m-1}$.
\end{lem}
\begin{proof}
  If $b=\sum_{k=r}^{R}c_{k}F_{k}$ is the Ostrowski representation of $b$, then Lemma~\ref{lem:2} implies $2/\ph^{m-2}>1/\ph^{r+1}$. From which we conclude that $r\geq m-4$. Since $2F_{m}=F_{m+1}+F_{m-2}$, by Lemma~\ref{lem:1}, the only possible Ostrowski representations of $b$ are the following:
  \begin{center}
  \begin{tabular}{rlll}
     $b =$ & $F_{m-1}$,                 & $F_{m}$,                 & $F_{m+1}$,\\
           & $F_{m-1}+F_{m-4}$,         & $F_{m}+F_{m-4}$,         & $F_{m+1}+F_{m-4}$,\\
           & $F_{m-1}+F_{m-3}$,         & $F_{m}+F_{m-3}$,         & $F_{m+1}+F_{m-3}$.\\
           &                            & $F_{m}+F_{m-2}$,         &\\
           &                            & $F_{m}+F_{m-2}+F_{m-4}$, &
  \end{tabular}
  \end{center}
  By assumption $r\geq m-4\geq 3$, and Lemma~\ref{lem:1} also implies that $[b\alpha]$ is obtained by replacing each term $F_{k}$ with $L_{k+1}$ in the list above. Using the fact that $L_{i}+L_{j}\equiv 0 \pmod{5}$ if and only if $i-j\equiv 2\pmod{4}$ and that $L_{m}\equiv 2L_{m+1}\equiv L_{m+4}\pmod{5}$, we find that the only values of $b$ above for which $[b\alpha]\equiv L_{m}\pmod{5}$ are
  \begin{equation*}
  b=F_{m-1}, F_{m}+F_{m-4}.
  \end{equation*}
  If $b=F_{m}+F_{m-4}$, then from Equation (\ref{eq:d2})
  \[\|b\alpha\|=|(L_{m+1}+L_{m-3})-(F_{m}+F_{m-4})\alpha|=\frac{1}{\ph^{m}}+\frac{1}{\ph^{m-4}}=\frac{\ph^4+1}{\ph^{m}}>\frac{2}{\ph^{m-2}}.\]
 Hence, $b=F_{m-1}$.
\end{proof}
\begin{prop}\label{prop:1b}
  Let $m\geq 7$, and let $n$ be an integer with $L_{m}\leq n<2L_{m+1}$ and $n\equiv L_{m}\pmod{5}$. Then all weight $n$ sums of the form
  \[s=a_1\z+a_2\z^2+a_3\z^3+a_4\z^4\in\Sigma_5^0(n)\]
  satisfy
  \[|s|\geq \frac{1}{\ph^{m-1}}.\]
  This bound is met with equality if and only if $n=L_m$ and $s$ or $\overline{s}$ is equal to
  \[F_{m}\z+F_{m-1}\z^3+F_{m-1}\z^4.\]
\end{prop}
\begin{proof}
  We can assume $|s|\leq 1/\ph^{m-1}<1/\ph^4$. Applying Lemma~\ref{lem:bnxy}, we have 
  \[|\re(s)|=\frac{1}{2\ph}|n-b\alpha|,\]
  where $b=a_2+a_3$ satisfies $b=[n/\alpha]$ and $[b\alpha]=n\equiv L_{m}\pmod{5}$; and
    \[|\im(s)|^2=\frac{4-\ph^2}{4}|x\ph-y|^2,\]
  where $x=a_1-a_4$, and $y=a_3-a_2$ satisfy $y=[x\ph]$ and $x=[y/\ph]$. 
  Combining $n\equiv L_{m}\pmod{5}$ with $L_{m}\leq n<2L_{m+1}$ gives $L_{m}\leq n\leq 2L_{m+1}-5$, so, applying Lemma~\ref{lem:0a}, we conclude $F_{m-1}\leq b <2F_{m}$. Additionally,
  \[\frac{1}{\ph^{m-1}}\geq |s|\geq \frac{1}{2\ph}|n-b\alpha|= \frac{1}{2\ph}\|b\alpha\|,\]
  implies $\|b\alpha\|\leq 2/\ph^{m-2}$. So, by Lemma~\ref{lem:1a}, we find $b=F_{m-1}$, from which $n=[F_{m-1}\alpha]=L_m$. Thus,
  \[|\re(s)|
  =\frac{1}{2\ph}|L_m-F_{m-1}\alpha|=\frac{1}{2\ph^m}.\]
  Since $b=a_2+a_3=F_{m-1}$ and $a_1+a_4=L_{m}-F_{m-1}=F_{m+1}$ cannot both be even numbers, $x=0$ implies $y\neq 0$. If $x=0$, then for $m\geq 3$
  \[|\im(s)|^2\geq \frac{4-\ph^2}{4}>\frac{1}{\ph^{2m-2}}.\]
  So we assume, without loss of generality, that $x>0$. Now,
  \[\frac{1}{\ph^{2m-2}}\geq |s|^2=\frac{1}{4\ph^{2m}}+\frac{4-\ph^2}{4}|x\ph-y|^2,\]
  implies $\|x\ph\|=|x\ph-y|\leq 1/\ph^{m-2}$. Let $x=\sum_{k=r}^R c_{k}F_{k}$ be the Ostrowski representation of $x$. Then, Lemma \ref{lem:2} gives $1/\ph^{m-2}\geq \|x\ph\|>1/\ph^{r+1}$. Hence, $r\geq m-2$, and, in particular, $x\geq F_{m-2}$. Therefore, $y\geq [F_{m-2}\ph]=F_{m-1}$ for $m\geq 4$. Combining $a_3+a_2=F_{m-1}$ with $y=a_3-a_2\geq F_{m-1}$, we find $a_2=0$ and $a_3=F_{m-1}$. This implies $y=F_{m-1}$, so $x=[F_{m-1}/\ph]=F_{m-2}$, for $m\ge 3$. From $a_1+a_4=F_{m+1}$ and $x=a_1-a_4=F_{m-2}$, we find $a_1=F_{m}$ and $a_{4}=F_{m-1}$. Therefore, the sum $s$ or its conjugate must be equal to $F_{m}\z+F_{m-1}\z^3+F_{m-1}\z^4$, in which case 
  \begin{align*}
    |s|&=|F_{m}+F_{m-1}\z^2+F_{m-1}\z^3|\\
    &=|F_{m}- F_{m-1}\ph|\\
    &=\frac{1}{\ph^{m-1}}.\qedhere
  \end{align*}
\end{proof}
Now we consider those weights $n$ satisfying $2L_m \leq n <L_{m+3}$ for some $m$.
\begin{lem}\label{lem:0a1} Let $m\geq 1$ be an integer. If $2L_{m}\leq n \leq L_{m+3}-5$, then
  \[2F_{m-1}\leq \left[\frac{n}{\alpha}\right]<F_{m+2}.\]
\end{lem}
\begin{proof} Analogous to the proof of Lemma~\ref{lem:0a}.
\end{proof}

\begin{lem}\label{lem:2a}
  Let $m\geq 6$, and let $b$ be an integer satisfying $2F_{m-1}\leq b<F_{m+2}$. If $[b\alpha]\equiv 2L_{m}\pmod{5}$ and
  \[\|b\alpha\|\leq \frac{2}{\ph^{m-1}},\]
  then $b=F_{m}+F_{m-3}$.
\end{lem}
\begin{proof}
The proof is analogous to the one of Lemma~\ref{lem:1a}. From Lemma~\ref{lem:2}, we have $2/\ph^{m-1}\geq\|b\alpha\|>1/\ph^{r+1}$. So, $r\geq m-3$, and when $m\geq 6$ the only possible value of $b$ satisfying $\|b\alpha\|\leq 2/\ph^{m-1}$ and $[b\alpha]\equiv 2L_{m}\pmod{5}$ is $b=F_{m}+F_{m-3}$.
\end{proof}

\begin{prop}\label{prop:2b} Let $m\geq 6$, and let $n$ be an integer satisfying $2L_{m}\leq n<L_{m+3}$ and $n\equiv 2L_{m}\pmod{5}$. If
  \[s=a_1\z+a_2\z^2+a_3\z^3+a_4\z^4\in\Sigma_5^0(n)\]
  is a weight $n$ sum, then
  \[|s|\geq \frac{1}{\ph^m}.\]
  Furthermore, this bound is met with equality if and only if $n=2L_m$ and
  \[s = F_{m+1}\z + F_{m-1}\z^2 + F_{m-1}\z^3 + F_{m+1}\z^4.\]
\end{prop}
\begin{proof} 
  We can assume
  \[|s|\leq \frac{1}{\ph^{\kappa(n)}}=\frac{1}{\ph^{m}}\leq \frac{1}{\ph^6}.\]
  Lemma~\ref{lem:bnxy} implies that $|\re(s)|=\frac{1}{2\ph}|n-b\alpha|$, where $b=a_2+a_3$ satisfies $n=[b\alpha ]$ and $b=[n/\alpha]$. So,
  \[\frac{1}{\ph^{m}}\geq |s|\geq \frac{1}{2\ph}\|b\alpha\|,\]
  which implies $\|b\alpha\|\leq 2/\ph^{m-1}$. Since $m\geq 6$, Lemma~\ref{lem:2a} implies $b=F_{m}+F_{m-3}=2F_{m-1}$. Thus, $n=[b\alpha]=2L_{m}$ and
  \[|\re(s)|=\frac{1}{\ph}|L_{m}- F_{m-1}\alpha|=\frac{1}{\ph^m}.\]
  From $1/\ph^m\geq |s|\geq |\re(s)|$, it follows that  $\im(s)=0$. Noting that
  \[\im(s)^2=\frac{4-\ph^2}{4}|(a_1-a_4)\ph - (a_3-a_2)|^2=0,\]
  we obtain $a_1=a_4$ and $a_2=a_3$. Hence, $a_2=a_3=F_{m-1}$ and $a_1=a_4=F_{m+1}$.
\end{proof}

\begin{proof}[Proof of Theorem~\ref{thm:1a}] 
  In view of Proposition~\ref{prop:red} it suffices to show that $|s|\geq 1/\ph^{\kappa(n)}$ for all \[s=a_1\z+a_2\z^2+a_3\z^3+a_4\z^4\in\Sigma_5^0(n).\]
  It can be verified computationally (e.g. by backtracking search on $\Sigma_5(n)$) that for all $n\leq \max\{L_{7},2L_6\}=36$ the inequality $\sigma_5(n)\geq 1/\ph^{\kappa(n)}$ holds. Now suppose $n\geq 36$ and that $s$ is a weight $n$ sum. Then $n\equiv L_{m}\pmod{5}$ and $n\in [L_{m},2L_{m+1})$ for some $m\geq 7$, or $n\equiv 2L_{m}\pmod{5}$ and $n\in [2L_{m},L_{m+3})$ for some $m\geq 6$. In the former case, we have that $\kappa(n)=m-1$, and by Proposition~\ref{prop:1b}, $|s|\geq 1/\ph^{m-1}=1/\ph^{\kappa(n)}$. In the latter case, we have that $\kappa(n)=m$ and by Proposition~\ref{prop:2b},  $|s|\geq 1/\ph^{m}=1/\ph^{\kappa(n)}$.
\end{proof}
Using Propositions \ref{prop:1b} and \ref{prop:2b} we can characterize the set of minimal sums. Suppose that $n\geq 36$, $n\in[L_m,2L_{m+1})$, $n\equiv L_m\pmod{5}$, and that $s\in\Sigma_5(n)$ satisfies $|s|=\sigma_5(n)=1/\ph^{m-1}$. After multiplying by an appropriate phase $\z^j$, we may assume
\[\zeta^{j} s=a_1\z+a_2\z^2+a_3\z^3+a_4\z^4\in\Sigma_5^0(n-5f),\]
for some integer $f\geq 0$. Thus, $1/\ph^{m-1}=|\z^js|\geq 1/\ph^{\kappa(n-5f)}$, and since $\kappa(n-5f)\leq \kappa(n)=m-1$, we find $\kappa(n-5f)=\kappa(n)$. Hence, by Proposition~\ref{prop:1b} we conclude that $s$ or $\overline{s}$ 
belongs to the set
\[
\mathcal{M}=
\{\z^h(F_m+F_{m-1}\z^2+F_{m-1}\z^3)
\mid 0\leq h\leq 4\}.\]
Since this set is closed under conjugation, we see that
$s$ actually belongs to $\mathcal{M}$.
Similarly, one can show that when $n\in [2L_m,L_{m+3})$, $n\equiv 2L_m\pmod{5}$, and $s\in\Sigma_5(n)$ satisfies $|s|=\sigma_5(n)=1/\ph^m$,
Proposition~\ref{prop:2b} implies that $s$ belongs to the set
\[\mathcal{M}'=\{\z^h(F_{m+1}\z+F_{m-1}\z^2+F_{m-1}\z^3+F_{m+1}\z^4)
\mid 0\leq h\leq 4\}.\]
An exhaustive search shows that the characterization of minimal sums above also holds true for $1\leq n <36$, with $n\not\equiv 0\pmod{5}$. 

\section{Proof of Theorem~\ref{thm:2a}}\label{sec:thm2}
Our proof of Theorem~\ref{thm:2a} follows the same structure as the proof of Theorem~\ref{thm:1a}. Here, we require the following identity, which holds for $m\geq 0$ and is easily shown by induction 
\begin{equation}\label{eq:1b3} 
  L_{m+1}-\ph L_{m} = (-1)^{m-1}\frac{\sqrt{5}}{\ph^{m}}.
\end{equation}
\begin{lem} \label{lem:ub2} Let $k\geq 1$ be an integer and let $n=5k$. Then
  \[\sigma_5(n)^2\leq \frac{\sqrt{5}}{\ph^{2\nu(k)-1}}.\]
\end{lem}
\begin{proof} There is a unique integer $m\geq 2$ such that $n$ lies in the interval $[5F_m,5F_{m+1})$, in which case $\nu(k)=m$. Write $n=5f+5F_m$ for some non-negative integer $f$ and consider the sum
    \begin{align*}
    s&=f(1+\z+\z^2+\z^3+\z^4)+2F_{m}\zeta+F_{m-2}\zeta^2+F_{m}\zeta^3+F_{m+1}\zeta^4\in\Sigma_5^0(n).
    \end{align*}
    Notice that
  \begin{align*}
    2\re(s)
    &=\frac{2F_{m}+F_{m+1}}{\ph} -(F_{m-2}+F_{m})\ph\\
    &=\frac{1}{\ph}\left(2F_{m}+F_{m+1}-(\ph+1)(F_{m-2}+F_{m})\right)\\
    &=\frac{1}{\ph}(L_{m}- L_{m-1}\ph)\\
    &=(-1)^{m}\frac{\sqrt{5}}{\ph^{m}}.\refby{eq:1b3}
  \end{align*}
  Observe also that
  \[s=F_{m+1}(\zeta+\zeta^4)+F_{m-2}(\zeta^2+\zeta^3)+F_{m-2}\zeta + F_{m-1}\zeta^3.\]
  Hence,
  \begin{align*}
 \im(s)&=\im(F_{m-2}\zeta+F_{m-1}\zeta^3)\\
    &=F_{m-2}\im(\zeta) - F_{m-1}\im(\zeta^2)\\
    &=\frac{\im(\zeta)}{\ph}\left(F_{m-2}\ph-F_{m-1}\right)\refby{eq:ratio}\\
    &=(-1)^{m-1}\frac{\sqrt{4\ph^2-1}}{2\ph^{m}}.
  \end{align*}
  Therefore,
  \[
    \sigma_5(n)^2\leq |s|^2
    = \frac{5}{4\ph^{2m}}+\frac{4\ph^2-1}{4\ph^{2m}}=
     \frac{\sqrt{5}}{\ph^{2m-1}}=\frac{\sqrt{5}}{\ph^{2\nu(k)-1}}.\qedhere
  \]
\end{proof}

\begin{prop}\label{prop:red1}
    The inequality
    \[\sigma_5^0(5k)^2\geq \frac{\sqrt{5}}{\ph^{2\nu(k)-1}}\]
    for all integers $k\geq 1$, implies Theorem~\ref{thm:2a}.
\end{prop}
\begin{proof}
    Analogous to the proof of Proposition~\ref{prop:red}.
\end{proof}

\begin{lem}\label{lem:1d1}
  Let $m\geq 2$ be an integer. If $5F_m\leq n\leq 5F_{m+1}-5$, then
  \[L_{m-1}\leq\left[\frac{n}{\alpha}\right]<L_{m}.\]
\end{lem}
\begin{proof} Notice that
  \begin{align*}
      \frac{5F_m}{\alpha} &= \frac{L_{m+1}+L_{m-1}}{\alpha}\\
      &=F_{m}+\frac{(-1)^{m}}{\ph^{m}\alpha}+F_{m-2}+\frac{(-1)^{m-2}}{\ph^{m-2}\alpha}\refby{eq:d2}\\
      &=L_{m-1}+\frac{(-1)^m}{\ph^m}.
  \end{align*}
  Using this identity and proceeding in analogy to the proofs of Lemmas \ref{lem:0a} and \ref{lem:0a1}, the result follows.
\end{proof}
\begin{lem}\label{lem:1e}
  Let $m\geq 7$ and $b$ be integers with $L_{m-1}\leq b < L_{m}$. If $[b\alpha]\equiv 0\pmod{5}$ and
  \[\|b\alpha\|^2\leq \frac{4\sqrt{5}}{\ph^{2m-3}},\]
  then $b=L_{m-1}$.
\end{lem}
\begin{proof} Let $b=\sum_{k=r}^{R}c_{k}F_k$ be the Ostrowski representation of $b$. Applying Lemma~\ref{lem:2}, $4\sqrt{5}/\ph^{2m-3}\geq \|b\alpha\|^2\geq 1/\ph^{2r+2}$, from which follows $r\geq m-4$. Noting that $L_{m-1}=F_{m}+F_{m-2}$, and proceeding analogously to the proof of Lemmas \ref{lem:1a} and \ref{lem:2a}, we find that, for $m\geq 7$, the only value of $b$ satisfying $\|b\alpha\|^2\leq 4\sqrt{5}/\ph^{2m-3}$ and $[b\alpha]\equiv 0\pmod{5}$ simultaneously is $b=F_{m}+F_{m-2}$.
\end{proof}


\begin{prop}\label{prop:1f} Let $m\geq 5$ and $n=5F_{m}$. If \[s=a_1\z+a_2\z^2+a_3\z^3+a_4\z^4\in\Sigma_5^0(n),\]
 and $a_2+a_3=L_{m-1}$, then
  \[|s|^2\geq \frac{\sqrt{5}}{\ph^{2m-1}}.\]
 Furthermore, $s$ satisfies the above bound with equality if and only if $s$ or $\overline{s}$ is equal to
  \[2F_m\z + F_{m-2}\z^2 +F_m\z^3+F_{m+1}\z^4.\]
\end{prop}
\begin{proof}
We can assume
  \[|s|^2\leq \frac{\sqrt{5}}{\ph^{2m-1}}\leq \frac{\sqrt{5}}{\ph^9}<\frac{1}{\ph^4}.\]
  Let $b:=a_2+a_3$. Applying Lemma~\ref{lem:bnxy}, and using the fact that $n=L_{m+1}+L_{m-1}$ and $b=L_{m-1}=F_{m}+F_{m-2}$, we have
  \begin{align*}
    |\re(s)|^2 &= \frac{1}{4\ph^2}|n-b\alpha|^2=\frac{5}{4\ph^{2m}},
  \end{align*}
  and
  \[|\im(s)|^2=\frac{4-\ph^2}{4}|x\ph-y|^2,\]
  where $x=a_1-a_4$, and $y=a_3-a_2$ satisfy $y=[x\ph]$ and $x=[y/\ph]$. Noting that $a_1+a_4=5F_{m}-b=L_{m+1}$, we can express the coefficients of $s$ in terms of $x$ and $y$ as
  \[(a_1,a_2,a_3,a_4)=\frac{1}{2}(L_{m+1}+x,L_{m-1}-y,L_{m-1}+y,L_{m+1}-x).\]
   Also, from $L_{m+1}\not\equiv L_{m-1}\pmod{2}$ we find that $x=0$ implies $y\neq 0$. If $x=0$, then for $m\geq 3$,
  \begin{align*}
    |s|^2\geq |\im(s)|^2\geq\frac{4-\ph^2}{4}>\frac{\sqrt{5}}{\ph^{2m-1}},
  \end{align*}
 Hence, without loss of generality, we may assume that $x>0$. Observe that
 \[\frac{\sqrt{5}}{\ph^{2m-1}}\geq |s|^2\geq \frac{5}{4\ph^{2m}}+\frac{4-\ph^2}{4}\|x\ph\|^2.\]
 So, it follows that $\|x\ph\|^2\leq 1/\ph^{2m-4}$. Let $x=\sum_{k=r}^R c_{k}F_{k}$ be the Ostrowski representation of $x$. Then, Lemma~\ref{lem:2} implies $\|x\ph\|^2>1/\ph^{2r+2}$, so $r\geq m-2$ and, in particular, $x\geq F_{m-2}$. From $0\leq 2a_2=L_{m-1}-y$, it follows that $y\leq L_{m-1}=F_{m}+F_{m-2}$, and then, for $m\geq 5$, $x=[y/\ph]\leq F_{m-1}+F_{m-3}$. Hence, the only valid Ostrowski representations of $x$ are $x= F_{m-2}$, or $x=F_{m-1}$. 
  
   Case 1:  $x=F_{m-1}$. Then  $y=F_{m}$, $2a_2=L_{m-1}-F_{m}=F_{m-2}$, and $2a_4=L_{m+1}-F_{m-1}=3F_{m}$. But this implies $F_{m-2}\equiv F_{m}\equiv 0\pmod{2}$, a contradiction.
  
  Case 2: $x=F_{m-2}$. Then $y=F_{m-1}$, and  $(a_1,a_2,a_3,a_4)=(2F_{m},F_{m-2},F_m,F_{m+1})$. The corresponding sum $s$ satisfies $|s|^2=\sqrt{5}/\ph^{2m-1}$, as shown in the proof of Lemma~\ref{lem:ub2}.
  %
\end{proof}

\begin{proof}[Proof of Theorem~\ref{thm:2a}]
  In view of Proposition~\ref{prop:red1} it suffices to show that for integers $n=5k$ satisfying $5F_{m}\leq n<5F_{m+1}$, for some $m\geq 2$, and any weight $n$ sum
  \[s=a_1\z+a_2\z^2+a_3\z^3+a_4\z^4\in\Sigma_5^0(n),\]
  we have
  \[|s|^2\geq \frac{\sqrt{5}}{\ph^{2m-1}}.\]
  For $n< 5F_{7}=65$ the claim can be verified computationally. Now suppose $n=5k\geq 65$, so that there is a unique integer $m\geq 7$ such that $5F_{m}\leq n<5F_{m+1}$, and thus $\nu(k)=m$. We may assume without loss of generality that
  \[|s|^2\leq\frac{\sqrt{5}}{\ph^{2m-1}}\leq \frac{\sqrt{5}}{\ph^{13}}<\frac{1}{\ph^4}.\]
  Hence, Lemma~\ref{lem:bnxy} implies that $n=[b\alpha]$ and $b=[n/\alpha]$, where $b=a_2+a_3$; and that $y=[x\ph]$ and $x=[y/\ph]$ where $x=a_1-a_4$ and $y=a_3-a_2$.
  The conditions $n\equiv 0\pmod{5}$ and $5F_{m}\leq n<5F_{m+1}$ imply $n\leq5F_{m+1}-5$, so, by Lemma~\ref{lem:1d1}, $L_{m-1}\leq b <L_{m}$. Since
  \[\frac{\sqrt{5}}{\ph^{2m-1}}\geq |s|^2\geq |\re(s)|^2=\frac{1}{4\ph^2}\|b\alpha\|^2,\]
  we find $\|b\alpha\|^2\leq 4\sqrt{5}/\ph^{2m-3}$, and $[b\alpha]=n\equiv 0\pmod{5}$. By Lemma~\ref{lem:1e}, we obtain $b=L_{m-1}=F_{m}+F_{m-2}$, and hence $n=L_{m+1}+L_{m-1}=5F_m$. Then, applying Proposition~\ref{prop:1f}, $|s|^2\geq \sqrt{5}/\ph^{2m-1}$.
\end{proof}

Similarly to the discussion at the end of Section \ref{sec:thm1}, one can show that Proposition~\ref{prop:1f} characterizes the set of minimal sums at orders $n\equiv 0\pmod{5}$. Now, for $n\geq 65$, using the monotonicity of $\nu$, this proposition implies that, whenever $n\equiv 0\pmod{5}$ and $n\in [5F_m, 5F_{m+1})$, if $s\in \Sigma_5(n)$ satisfies $|s|=\sigma_5(n)$, then $s$ belongs to the set $\mathcal{M}''\cup\overline{\mathcal{M}''},$ where
\[\mathcal{M}''=\{\zeta^h(2F_m\z + F_{m-2}\z^2 +F_m\z^3+F_{m+1}\z^4)\mid 0\leq h\leq 4\}.\]
 An exhaustive computation shows that the same characterization of minimal sums above also holds for $5\leq n \leq 60$, with $n\equiv 0\pmod{5}$.

\subsection*{Acknowledgments}
AM was supported by the Japanese Society for the Promotion of Science (JSPS) KAKENHI Grant Number 25K07095. GNP was supported by the JSPS Postdoctoral fellowship, and by the JSPS KAKENHI Grant Number 24KF0176.

\printbibliography

\end{document}